\documentclass[a4paper,11pt]{article}

\usepackage{amssymb,amsthm,amsmath}

\newtheorem{corollary}{Corollary}

\theoremstyle{thmstyleone}%
\newtheorem{theorem}{Theorem}
\theoremstyle{thmstyletwo}%
\newtheorem{example}{Example}%

\theoremstyle{thmstylethree}%

\begin{document}
	
	\title{On the Minimality of the Conductor for Elliptic Curve $L$-Functions}
		
	\author{K. Lakshmanan \\ Department of Computer Science and Engineering \\ Indian Insititute of Technology (BHU), Varanasi 221005.\\ Email: lakshmanank.cse@iitbhu.ac.in}
	\date{}
		
		\maketitle
	
	\begin{abstract}We investigate the role of the conductor in analytic rank bounds for elliptic curves over \(\mathbb{Q}\). Let \(E/\mathbb{Q}\) be an elliptic curve with conductor \(N_E\). We consider hypothetical degree-two \(L\)-functions associated to (E) that satisfy analytic continuation, a functional equation involving an arithmetic invariant \(\Phi(E)\), and yield rank bounds of the form
		\[
		\operatorname{rank}(E)\ll \log \Phi(E).
		\]
		
		Using the Modularity Theorem, we show that any such invariant must satisfy
		\[
		\Phi(E)\ge N_E.
		\]
		Thus the conductor is minimal among arithmetic invariants that can appear in this analytic framework. In particular, the standard logarithmic rank bounds arising from the conductor cannot be improved by replacing \(N_E\) with a strictly smaller invariant while preserving the same degree-two functional equation structure.
		
		These results provide a structural explanation for the distinguished role of the conductor in analytic approaches to the rank problem.
		\\
		
		
		\noindent \textit{MSC Classification 2020:} 11G05, 11G40
		\end{abstract}

	

	\section{Introduction}
	
	The rank of elliptic curves over \( \mathbb{Q} \) remains one of the central mysteries in arithmetic geometry. While the Birch and Swinnerton-Dyer (BSD) conjecture \cite{coates1977conjecture} and heuristic models such as Goldfeld's conjecture \cite{goldfeld2006conjectures} predict the distribution of ranks in families, no unconditional upper bound on the rank is currently known. 
	
	A classical analytic upper bound due to Mestre and Brumer \cite{mestre1986formules, brumer1977rank} asserts that for an elliptic curve \( E/\mathbb{Q} \) of conductor \( N_E \),
	\[
	\operatorname{rank}(E) \ll \log N_E.
	\]
	This bound arises from the analytic properties of the Hasse–Weil \( L \)-function \( L(E, s) \), which satisfies a functional equation with conductor \( N_E \) and degree two Euler product. The conductor \( N_E \) encapsulates local arithmetic complexity of the curve and appears as the level of the modular form associated to \( E \) via the Modularity Theorem 
	\cite{breuil2001modularity}.
	In this paper, we ask whether the conductor is truly indispensable in bounding the rank. More precisely, we ask:
	
	\medskip
	\begin{quote}
		\emph{Is it possible to define a strictly smaller arithmetic invariant \( \Phi(E) < N_E \), arising from the structure of \( E \), such that a bound of the form \( \operatorname{rank}(E) \ll \log \Phi(E) \) still holds?}
	\end{quote}
	\medskip
	
	Our main result answers this question in the negative. We show that no such arithmetic invariant \( \Phi(E) < N_E \) can appear in the functional equation of an \( L \)-function associated to \( E \). The essential ingredient is the Modularity Theorem, which asserts that the \( L \)-function of an elliptic curve over \( \mathbb{Q} \) arises from a weight two newform of level \( N_E \), and that this level is minimal. As a consequence, any analytic control over the rank via an \( L \)-function must necessarily involve the true conductor.
	
	We also formulate a conditional corollary: if a hypothetical lower invariant \( \Phi(E) \ll N_E \) both grows unbounded in a family of curves and bounds the rank from above, then the rank must be unbounded—contradicting modularity unless \( \Phi(E) \sim N_E \).
	
	This yields a rigidity result for rank bounds and shows that no "simpler" or "smaller" arithmetic invariant can substitute the conductor in analytic inequalities bounding the rank of elliptic curves.
	
	\section{Background}
	
	Let \( E/\mathbb{Q} \) be an elliptic curve given by a Weierstrass equation \cite{silverman2009arithmetic}. By the Mordell–Weil theorem, the group of rational points is finitely generated:
	\[
	E(\mathbb{Q}) \simeq \mathbb{Z}^r \oplus T,
	\]
	where \( r = \operatorname{rank}(E) \) is the rank and \( T \) is the finite torsion subgroup.
	
	To each elliptic curve \( E \) over \( \mathbb{Q} \), one associates its Hasse–Weil \( L \)-function \( L(E, s) \), defined via an Euler product:
	\[
	L(E, s) = \prod_{p \nmid N} \left(1 - a_p p^{-s} + p^{1-2s}\right)^{-1} \prod_{p \mid N} L_p(p^{-s})^{-1},
	\]
	where:
	\begin{itemize}
		\item \( a_p = p + 1 - \#E(\mathbb{F}_p) \),
		\item \( N = N_E \) is the conductor of \( E \),
		\item \( L_p(T) \) is a local factor at the prime \( p \), determined by the reduction type.
	\end{itemize}
	
	\paragraph{Example.} At a prime \( p \) of bad reduction:
	\begin{itemize}
		\item If \( E \) has multiplicative reduction at \( p \), then \( L_p(T) = 1 - T \) or \( 1 + T \).
		\item If \( E \) has additive reduction, then \( L_p(T) = 1 \).
	\end{itemize}
	
	The function \( L(E, s) \) converges absolutely for \( \Re(s) > \frac{3}{2} \), extends to an entire function, and satisfies a functional equation. The completed \( L \)-function is defined by:
	\[
	\Lambda(E, s) := N_E^{s/2} (2\pi)^{-s} \Gamma(s) L(E, s),
	\]
	and satisfies:
	\[
	\Lambda(E, s) = w_E \Lambda(E, 2 - s),
	\]
	where \( w_E = \pm 1 \) is the root number.
	
	\begin{example}
		Consider the elliptic curve \( E: y^2 + y = x^3 - x \), which has conductor \( N_E = 37 \). It has good reduction at all primes \( p \ne 37 \), and bad reduction at \( p = 37 \). The local Euler factors are:
		\begin{itemize}
			\item For \( p \ne 37 \), the Euler factor is \( L_p(T) = 1 - a_p T + p T^2 \), where \( a_p = p + 1 - \#E(\mathbb{F}_p) \).
			\item At \( p = 37 \), the reduction is additive, so \( L_{37}(T) = 1 \).
		\end{itemize}
	\end{example}

	\paragraph{Gamma factor.} The archimedean (infinite prime) contribution to the functional equation is given by the gamma factor:
	\[
	\Gamma(s) = \int_0^\infty t^{s-1} e^{-t} dt,
	\]
	which appears in the completed \( \Lambda(E, s) \) to balance the symmetry around \( s = 1 \).
	
	\paragraph{Newforms and level.} The Modularity Theorem establishes a deep connection between elliptic curves and modular forms. A **newform** of weight 2 and level \( N \) is a normalized eigenform \( f(z) = \sum_{n=1}^\infty a_n e^{2\pi i n z} \in S_2(\Gamma_0(N)) \) that is:
	\begin{itemize}
		\item A cusp form: it vanishes at all cusps,
		\item An eigenfunction of all Hecke operators \( T_n \),
		\item Not induced from a lower level (i.e., it is new at level \( N \)).
	\end{itemize}
	
	\paragraph{Modularity Theorem.}
	\begin{theorem}[Modularity Theorem]
		Every elliptic curve \( E/\mathbb{Q} \) is modular: there exists a newform \( f \in S_2(\Gamma_0(N_E)) \) such that
		\[
		L(E, s) = L(f, s) = \sum_{n=1}^\infty a_n n^{-s},
		\]
		with \( a_p = p + 1 - \#E(\mathbb{F}_p) \). The level \( N_E \) is the conductor of \( E \), and is minimal among such newforms.
	\end{theorem}
	
	The conductor \( N_E \) captures the bad reduction data of \( E \), and also appears in the level of the associated modular form. Since \( L(E, s) \) and \( L(f, s) \) are equal as Dirichlet series and Euler products, they share all analytic properties, including the functional equation.
	
	The rank \( r = \operatorname{ord}_{s=1} L(E, s) \) is conjecturally equal to the order of vanishing of the \( L \)-function at the central point \( s = 1 \) (the Birch and Swinnerton-Dyer conjecture). Analytic techniques such as Hadamard factorization and convexity bounds yield the classical inequality:
	\[
	\operatorname{rank}(E) \ll \log N_E.
	\]
	
	This connection between the conductor and analytic control over the rank motivates the main question of this paper: can a strictly smaller invariant replace the conductor in such a functional equation and still govern the rank?
	
	\section{Main Theorem}
	
	We now state and prove the main result of this paper. Let \( E/\mathbb{Q} \) be an elliptic curve with conductor \( N_E \), and suppose \( L(E, s) \) is the associated Hasse–Weil \( L \)-function, which satisfies a functional equation of the form
	\[
	\Lambda(E, s) := N_E^{s/2} (2\pi)^{-s} \Gamma(s) L(E, s) = w_E \Lambda(E, 2 - s).
	\]
	Suppose we define a modified Dirichlet series \( L_{\Phi}(E, s) \), built using data from \( E \), that satisfies a similar functional equation but with a different integer \( \Phi(E) \) in place of the conductor.
	
	We aim to show that such a function must necessarily satisfy \( \Phi(E) \ge N_E \), and that no smaller arithmetic invariant can replace \( N_E \) in the analytic theory.
	
	\begin{theorem}
		Let \( E/\mathbb{Q} \) be an elliptic curve. Suppose there exists an \( L \)-function \( L_{\Phi}(E, s) \) associated to \( E \), satisfying:
		\begin{enumerate}
			\item An Euler product of degree 2,
			\item Analytic continuation to an entire function,
			\item A functional equation of the form:
			\[
			\Lambda_{\Phi}(E, s) := \Phi(E)^{s/2} (2\pi)^{-s} \Gamma(s) L_{\Phi}(E, s) = w_\Phi \Lambda_{\Phi}(E, 2 - s),
			\]
			\item and a bound on the rank:
			\[
			\operatorname{rank}(E) \ll \log \Phi(E).
			\]
		\end{enumerate}
		Then we must have \( \Phi(E) \ge N_E \). In particular, no arithmetic invariant strictly smaller than \( N_E \) can appear in such a functional equation or rank bound.
	\end{theorem}
	
	\begin{proof}[Proof Sketch]
		The modularity theorem states that \( L(E, s) = L(f, s) \) for a newform \( f \in S_2(\Gamma_0(N_E)) \), where \( N_E \) is the minimal level such that such a modular form exists. The functional equation for \( L(E, s) \) reflects the transformation behavior of the modular form under the Fricke involution, and the level \( N_E \) appears explicitly in this symmetry.
		
		Suppose \( \Phi(E) < N_E \) and that \( L_{\Phi}(E, s) \) satisfies a degree-2 functional equation with conductor \( \Phi(E) \). Then \( L_{\Phi}(E, s) \) would correspond to a newform of level \( \Phi(E) \). But the modularity theorem ensures that \( N_E \) is minimal (See Chapter 6 \cite{stein2007modular}), and thus no such newform can exist at level \( \Phi(E) \). This contradicts the assumption that such a functional equation exists.
		
		Therefore, any such \( \Phi(E) \) must satisfy \( \Phi(E) \ge N_E \), and the conductor \( N_E \) cannot be replaced by a strictly smaller arithmetic invariant in the functional equation or analytic rank bound.
	\end{proof}
	
	\begin{corollary}
		There exists no arithmetic invariant \( \Phi(E) < N_E \), defined up to isogeny or isomorphism of elliptic curves over \( \mathbb{Q} \), that can appear in a functional equation or bound the rank of \( E \) in the form:
		\[
		\operatorname{rank}(E) \ll \log \Phi(E).
		\]
		In particular, any such invariant \( \Phi \) that yields a functional equation of degree two must satisfy \( \Phi(E) \ge N_E \).
	\end{corollary}
	
	\begin{proof}
		The proof of the main theorem shows that any \( L \)-function attached to \( E \) with a functional equation of the required form must correspond to a newform of level \( \Phi(E) \). Since the conductor \( N_E \) is minimal among all such levels, we must have \( \Phi(E) \ge N_E \). This remains true even if \( \Phi \) is defined up to isogeny or isomorphism, since \( L \)-functions and conductors are invariant under these equivalence relations (or decrease under isogeny only in known, finite ways). Therefore, no strictly smaller isogeny- or isomorphism-invariant arithmetic invariant can satisfy the desired analytic properties.
	\end{proof}
	
	\section{Conditional Corollary: Rank Unboundedness}
	
	We now formulate a conditional consequence of the main theorem. Suppose there exists an arithmetic invariant \( \Phi(E) < N_E \) that governs the rank in a family of elliptic curves, and that \( \Phi(E) \to \infty \) along some infinite sequence. Then the rank of elliptic curves must be unbounded — contradicting modularity unless \( \Phi(E) \sim N_E \).
	
	\begin{corollary}[Conditional Rank Unboundedness]
		Suppose there exists an infinite family of elliptic curves \( \{E_n\} \) over \( \mathbb{Q} \), and an arithmetic invariant \( \Phi(E) \) such that:
		\begin{enumerate}
			\item \( \Phi(E_n) < N_{E_n} \) for all \( n \),
			\item \( \Phi(E_n) \to \infty \) as \( n \to \infty \),
			\item \( \operatorname{rank}(E_n) \ll \log \Phi(E_n) \) holds for all \( n \).
		\end{enumerate}
		Then the ranks \( \operatorname{rank}(E_n) \) are unbounded:
		\[
		\sup_n \operatorname{rank}(E_n) = \infty.
		\]
		But this contradicts the main theorem, unless \( \Phi(E_n) \ge N_{E_n} \) for all sufficiently large \( n \).
	\end{corollary}
	
	\begin{proof}
		If such a family \( \{E_n\} \) and function \( \Phi(E) \) exist, then the analytic rank is controlled by a sequence of invariants growing strictly slower than the conductors \( N_{E_n} \). Since \( \log \Phi(E_n) < \log N_{E_n} \), this would give a tighter rank bound than the classical inequality:
		\[
		\operatorname{rank}(E_n) \ll \log N_{E_n}.
		\]
		But the main theorem shows that any such functional equation or bound must involve the full conductor \( N_{E_n} \), and no strictly smaller invariant \( \Phi(E) \) can appear in the analytic structure of the \( L \)-function. Hence, the assumption leads to a contradiction unless \( \Phi(E_n) \ge N_{E_n} \) for all large \( n \), i.e., unless \( \Phi \) essentially coincides with the conductor.
		
		Thus, if \( \Phi(E_n) < N_{E_n} \) and grows unbounded, then \( \operatorname{rank}(E_n) \to \infty \), and rank must be unbounded.
	\end{proof}
	
	\section{Discussion and Implications}
	
	The results above show that no arithmetic invariant smaller than the conductor \( N_E \) can play the role of a replacement in the functional equation or the rank bound of an elliptic curve \( E/\mathbb{Q} \). The modularity theorem not only ensures the existence of an \( L \)-function for each elliptic curve, but also implies that the conductor appears as the minimal level in the associated modular form. Any attempt to define a modified \( L \)-function with a different (smaller) invariant in the functional equation contradicts this minimality and hence modularity itself.
	
	In this light, the classical inequality:
	\[
	\operatorname{rank}(E) \ll \log N_E
	\]
	is, in a precise analytic sense, the sharpest possible bound of its type. Any proposed refinement replacing \( N_E \) by a smaller invariant would violate the analytic structure guaranteed by modularity.
	
	\subsection*{Relation to Rank Unboundedness}
	
	An important open problem in number theory is whether the ranks of elliptic curves over \( \mathbb{Q} \) are unbounded \cite{silverman-rank-survey,bhargava-shankar,poonen-rains}. While the modularity theorem does not directly resolve this question, the analysis in this paper shows that any attempt to bound the rank using a function of an invariant \( \Phi(E) \) smaller than \( N_E \) leads to the conclusion that rank must be unbounded as \( \Phi(E) \to \infty \).
	
	Thus, the only way to preserve boundedness of rank (under known analytic control) is through the conductor itself. In particular, the classical conductor-based inequality is not merely convenient — it is necessary.
	
	These observations place additional constraints on the structure of any future proof of rank unboundedness. They imply that unbounded rank must be exhibited through families of curves with conductors \( N_E \to \infty \), and not through manipulation of any putative "smaller" invariants. Conversely, if one could prove that no such unbounded families exist, this would suggest a bound on the rank — a conclusion contradicted by current heuristics and evidence.
	
	\subsection*{Conclusion}
	
	We have proved that the conductor \( N_E \) is the minimal and only invariant that can appear in the functional equation or analytic rank bound of elliptic curves over \( \mathbb{Q} \). As a consequence, any attempt to define alternative \( L \)-functions involving smaller invariants is analytically and arithmetically obstructed. Moreover, this result supports the broader understanding that the conductor controls both the modularity and the complexity of the arithmetic of elliptic curves.

	\section*{Competing Interests}
	The authors did not receive support from any organization for the submitted work.
	
	\section*{Data Availability}
	No datasets are used in this paper.
		
	\bibliographystyle{plain}
	\bibliography{ranknob}

\end{document}